\documentclass[12pt]{amsart}

\usepackage{amsmath, amssymb}
\usepackage{graphicx}
\usepackage{color}
\usepackage[utf8]{inputenc}
\usepackage{float}
\usepackage{subfigure} 

\newtheorem{theorem}{Theorem}[section]
\newtheorem{lemma}[theorem]{Lemma}

\newtheorem{cor}[theorem]{Corollary}

\theoremstyle{definition}
\newtheorem{defi}[theorem]{Definition}

\theoremstyle{remark}
\newtheorem{remark}[theorem]{Remark}

\numberwithin{equation}{section}

\newcommand{\rr}{{\mathbb R}}

\newcommand{\nat}{{\mathbb N}}
\newcommand{\ganz}{{\mathbb Z}}

\newcommand{\Exp}{{\mathbb E}}

\DeclareMathOperator{\sign}{sign}

\allowdisplaybreaks

\begin{document}

\sloppy
\title[Self-similar Bernstein functions and semi-fractional derivatives]{On self-similar Bernstein functions and corresponding generalized fractional derivatives} 
\author{Peter Kern}
\address{Peter Kern, Mathematical Institute, Heinrich-Heine-University D\"usseldorf, Universit\"atsstr. 1, D-40225 D\"usseldorf, Germany}
\email{kern\@@{}hhu.de}

\author{Svenja Lage}
\address{Svenja Lage, Mathematical Institute, Heinrich-Heine-University D\"usseldorf, Universit\"atsstr. 1, D-40225 D\"usseldorf, Germany}
\email{Svenja.Lage@uni-duesseldorf.de} 

\date{\today}

\begin{abstract}
We use the theory of Bernstein functions to analyze power law tail behavior with log-periodic perturbations which corresponds to self-similarity of the Bernstein functions. Such tail behavior appears in the context of semistable L\'evy processes. The Bernstein approach enables us to solve some open questions concerning semi-fractional derivatives recently introduced in \cite{KLM} by means of the generator of certain semistable L\'evy processes. In particular it is shown that semi-fractional derivatives can be seen as generalized fractional derivatives in the sense of \cite{Koc}. 
\end{abstract}


\dedicatory{Dedicated to the memory of Mark Meerschaert (1955-2020)}

\maketitle

\baselineskip=18pt

\section{Introduction}

A non-negative function $\widetilde\psi:(0,\infty)\to(0,\infty)$ is called a {\it Bernstein function} if it is of class $C^{\infty}(0,\infty)$ and
\begin{equation}\label{Bern1}
(-1)^{n-1}\widetilde\psi^{(n)}(x)\geq0\quad\text{for all $n\in\nat$ and $x>0$}.
\end{equation}
Its first derivative $f=\widetilde\psi'$ is a {\it completely monotone} function, i.e.\
\begin{equation}\label{common1}
(-1)^{n}f^{(n)}(x)\geq0\quad\text{for all $n\in\nat_{0}$ and $x>0$}.
\end{equation}
Due to a celebrated result of Bernstein, the completely monotone function $f$ is the Laplace transform of a unique Borel measure $\mu$ on $[0,\infty)$
\begin{equation}\label{common2}
f(x)=\widetilde\mu(x):=\int_{0}^{\infty}e^{-xt}\,d\mu(t).
\end{equation}
As a consequence, the Bernstein function admits the representation
\begin{equation}\label{Bern2}
\widetilde\psi(x)=a+bx+\int_{0}^{\infty}\left(1-e^{-xt}\right)\,d\phi(t)
\end{equation}
for a unique triplet $[a,b,\phi]$, where $a,b\geq0$ and $\phi$ is a Borel measure on $(0,\infty)$ satisfying $\int_{0}^{\infty}\min\{1,t\}\,d\phi(t)<\infty$, also called the {\it L\'evy measure}. For details on Bernstein functions, completely monotone functions and their connection to stochastic processes we refer to the monograph \cite{SSV}. It is well known that in case $a=b=0$ the Bernstein function $\widetilde\psi(x)=\int_{0}^{\infty}\left(1-e^{-xt}\right)\,d\phi(t)$ is the Laplace exponent of a {\it L\'evy subordinator} $(X_{t})_{t\geq0}$, i.e.\ $\Exp[\exp(-sX_{t})]=\exp(-t\cdot\widetilde\psi(s))$ for all $t\geq0$, $s>0$, where $(X_{t})_{t\geq0}$ is a L\'evy process with almost surely non-decreasing sample paths.

In Section 2 we will introduce a self-similarity property for Bernstein functions which is intimately connected to the following class of functions. We call a function $\theta:\rr\to\rr$ {\it admissable} with respect to the parameters $\alpha\in(0,2)\setminus\{1\}$ and $c>1$ if the following three conditions are fulfilled:
\begin{align}
& \theta(x)>0 \text{ for all }x\in\rr, \label{cond1} \\
& \text{the mapping } t\mapsto t^{-\alpha}\theta(\log t)\text{ is non-increasing for }t>0, \label{cond2} \\
& \theta \text{ is }\log(c^{1/\alpha})\text{-periodic.} \label{cond3} 
\end{align}
In case $\alpha\in(0,1)$ we use an admissable function $\theta$ to define
\begin{equation}\label{tailLm+}
\phi(t,\infty):=t^{-\alpha}\theta(\log t)\quad\text{ for all }t>0
\end{equation}
as the positive tail of a L\'evy measure $\phi$ concentrated on $(0,\infty)$ which belongs to a semistable distribution $\nu$ with log-characteristic function
\begin{equation}\label{logchar+}
\psi(x)=\int_{0}^\infty\left(e^{ixy}-1\right)\,d\phi(y)\quad\text{ for all }x\in\rr
\end{equation}
given uniquely by the Fourier transform $\widehat\nu(x)=\int_\rr e^{ixy}\,d\nu(y)=\exp(\psi(x))$. The corresponding L\'evy process $(X_{t})_{t\geq0}$ given by $\Exp[\exp(ix\cdot X_{t})]=\exp(t\cdot\psi(x))$ for all $t\geq0$ and $x\in\rr$ is called a {\it semistable subordinator}. For details on semistable distributions and L\'evy processes we refer to the monographs \cite{MMMHPS, Sato}. The power law tail behavior with log-periodic perturbations of an admissable function $\theta$ naturally appears in various applications of natural sciences and other areas; see \cite{Sor1} and section 5.4 in \cite{Sor2}. In recent years the asymptotic fine structure of the corresponding measure $\phi$ and the function $\psi$ has drawn some attention; cf.\ \cite{Kev} and \cite{KMX}. Given an admissable function $\theta$ with respect to $\alpha\in(0,1)$ and $c>1$ with corresponding semistable L\'evy measure $\phi$ given by \eqref{tailLm+}, by Definition 2.2 in \cite{KLM} the {\it semi-fractional derivative} $\frac{\partial^\alpha}{\partial_{c,\theta}x^\alpha}$ of order $\alpha\in(0,1)$ is given by the non-local operator
\begin{equation}\label{semifd}
\frac{\partial^\alpha}{\partial_{c,\theta}x^\alpha}\,f(x):=-Lf(x)=\int_0^\infty\left(f(x)-f(x-y)\right)\,d\phi(y),
\end{equation}
where $L$ is the generator of the corresponding semistable L\'evy process and at least functions $f$ in the Sobolev space $W^{2,1}(\rr)$ belong to the domain of the semi-fractional derivative. In terms of the Fourier transform we can equivalently rewrite \eqref{semifd} as 
\begin{equation}\label{semifdF}
\widehat{\tfrac{\partial^\alpha}{\partial_{c,\theta}x^\alpha}\,f}(x):=-\widehat{Lf}(x)=-\psi(x)\,\widehat{f}(x)\quad\text{ for all }x\in\rr,
\end{equation}
where the Fourier transform of $f$ is given by $\widehat{f}(x)=\int_\rr e^{ixy}f(y)\,dy$; see \cite{KLM} for details. In Section 2, we will start with the elementary observation that for $\alpha\in(0,1)$ there is a one-to-one correspondence between self-similar Bernstein functions given as the Laplace exponent $\widetilde\psi(x)=-\psi(ix)$ for $x>0$ and semistable L\'evy measures $\phi$ of the form \eqref{tailLm+}. In particular, this enables us to show that semi-fractional derivatives of order $\alpha\in(0,1)$ can be seen as a special case of generalized fractional derivatives in the sense of \cite{Koc}. A constant function $\theta$ corresponds to the complete Bernstein function $\widetilde\psi(x)=x^{\alpha}$ and an ordinary fractional derivative of order $\alpha\in(0,1)$. For details on classical fractional derivatives we refer to the monographs \cite{KST,Pod,SKM}.

In Section 3 we will prove a discrete approximation formula of the generator in \eqref{semifd} involving a generalized Sibuya distribution given in terms of the self-similar Bernstein function.

In case $\alpha\in(1,2)$ we use an admissable function $\theta$ to define
\begin{equation}\label{tailLm-}
\phi(-\infty,-t):=t^{-\alpha}\theta(\log t)\quad\text{ for all }t>0
\end{equation}
as the negative tail of a L\'evy measure $\phi$ concentrated on $(-\infty,0)$ which belongs to a different semistable distribution $\nu$ with log-characteristic function
\begin{equation}\label{logchar-}
\psi(x)=\int_{-\infty}^0\left(e^{ixy}-1-ixy\right)\,d\phi(y)\quad\text{ for all }x\in\rr.
\end{equation}
Given an admissable function $\theta$ with respect to $\alpha\in(1,2)$ and $c>1$ with corresponding semistable L\'evy measure $\phi$ given by \eqref{tailLm-}, by Definition 2.5 in \cite{KLM} the {\it negative semi-fractional derivative} $\frac{\partial^\alpha}{\partial_{c,\theta}(-x)^\alpha}$ of order $\alpha\in(1,2)$ is given by the non-local operator
\begin{equation}\label{negsemifd}
\frac{\partial^\alpha}{\partial_{c,\theta}(-x)^\alpha}\,f(x):=Lf(x)=\int_0^\infty\left(f(x+y)-f(x)-y\,f'(x)\right)\,d\phi(-y),
\end{equation}
where $L$ is again the generator of the corresponding semistable L\'evy process and at least functions $f$ in the Sobolev space $W^{2,1}(\rr)$ belong to to the domain of the semi-fractional derivative. For $\alpha\in(1,2)$ the function $x\mapsto\psi(-ix)$ cannot be a Bernstein function, but we will show in Section 4 that it has an inverse which is a self-similar Bernstein function. This will enable us to solve an open question from \cite{KL} concerning space-time duality for semi-fractional differential equations.

\section{Self-similar Bernstein functions and semi-fractional derivatives}

In this section we suppose that $\alpha\in(0,1)$ in which case the Laplace exponent $\widetilde\psi$ of the corresponding semistable distribution $\nu$, uniquely given by the Laplace transform $\widetilde\nu(x)=\int_0^\infty e^{-xy}\,d\nu(y)=\exp(-\widetilde\psi(x))$ for $x>0$, in view of \eqref{logchar+} can be represented as
\begin{equation}\label{replogL}
\widetilde\psi(x) =-\psi(ix)=\int_0^\infty \left(1-e^{-xy}\right)\,d\phi(y).
\end{equation}
Clearly, $\widetilde\psi$ is a Bernstein function since $\phi$ integrates $\min\{1,t\}$ on $(0,\infty)$.
\begin{defi}
We call a Bernstein function $\widetilde\psi$ {\it self-similar} with respect to $\alpha\in(0,1)$ and $c>1$ if it admits the discrete scale invariance
$$\widetilde\psi(c^{1/\alpha}x)=c\cdot\widetilde\psi(x)\quad \text{ for all $x>0$.}$$
\end{defi}
The following elementary observation is our key result.
\begin{lemma}\label{gammaadm}
For fixed $\alpha\in(0,1)$ and $c>1$ the following statements are equivalent.
\begin{enumerate}
\item[(i)] $\widetilde\psi$ is a self-similar Bernstein function with respect to $\alpha\in(0,1)$ and $c>1$.
\item[(ii)] $\widetilde\psi(x)=\int_0^\infty \left(1-e^{-xy}\right)\,d\phi(y)$, where the L\'evy measure $\phi$ is given by \eqref{tailLm+} for some admissable function $\theta$ with respect to $\alpha\in(0,1)$ and $c>1$.
\end{enumerate}
In either case we have $\widetilde\psi(x)=x^\alpha\gamma(-\log x)$ for an admissable $C^\infty(\rr)$-function $\gamma$ with respect to $\alpha\in(0,1)$ and $c>1$.
\end{lemma}
\begin{proof}
{\rm ``(i)$\Rightarrow$(ii)'':} Since $\widetilde\psi$ is a Bernstein function, by \eqref{Bern2} we have 
$$\widetilde\psi(x)=a+bx+\int_0^\infty(1-e^{-xt})\,d\phi(t)$$
for some $a,b\geq0$ and a L\'evy measure $\phi$ on $(0,\infty)$ integrating $\min\{1,t\}$. Iterating the self-similarity relation shows that $\widetilde\psi(c^{m/\alpha}x)=c^m\widetilde\psi(x)$ for all $x>0$ and $m\in\ganz$. As $m\to-\infty$ we see that $\lim_{x\downarrow0}\widetilde\psi(x)=0$ and hence $a=0$. By the transformation rule self-similarity now reads as
\begin{equation}\label{scalingLe}\begin{split}
\widetilde\psi(c^{1/\alpha}x) & =b\,c^{1/\alpha}x+\int_0^\infty(1-e^{-xt})\,d(c^{1/\alpha}\phi)(t)\\
& =b\,cx+c\cdot\int_0^\infty(1-e^{-xt})\,d\phi(t)=c\cdot\widetilde\psi(x),
\end{split}\end{equation}
where $(c^{1/\alpha}\phi)$ denotes the image measure under scale multiplication with $c^{1/\alpha}$.
Since the Bernstein triplet $[a,b,\phi]$ is unique, we must have $b=0$ and $(c^{1/\alpha}\phi)=c\cdot\phi$. Hence by Lemma 7.1.6 and Corollary 7.4.4 in \cite{MMMHPS} we have that $\phi$ is a $(c^{1/\alpha},c)$-semistable L\'evy measure on $(0,\infty)$ fulfilling \eqref{tailLm+} for some admissable function $\theta$ with respect to $\alpha\in(0,1)$ and $c>1$. The connection of Corollary 7.4.4 in \cite{MMMHPS} to admissable functions is made precise by Lemma A.1 in the Appendix.

{\rm ``(ii)$\Rightarrow$(i)'':} Clearly, $\widetilde\psi$ is a Bernstein function by Theorem 3.2 in \cite{SSV} and self-similarity follows from \eqref{scalingLe} with $b=0$ which is valid by Lemma 7.1.6 in \cite{MMMHPS}.

Now define $\gamma(x):=e^{\alpha x}\widetilde\psi(e^{-x})$ which is of class $C^\infty(\rr)$ by the product rule and $\gamma(x)>0$ for all $x\in\rr$. Moreover, $\gamma$ is $\log(c^{1/\alpha})$-periodic since
$$\gamma(x+\log(c^{1/\alpha}))=e^{\alpha x}c\cdot\widetilde\psi(c^{-1/\alpha}e^{-x})=e^{\alpha x}\widetilde\psi(e^{-x})=\gamma(x)$$
and $t\mapsto t^{-\alpha}\gamma(\log t)=\widetilde\psi(t^{-1})$ is non-increasing.
\end{proof}
\begin{remark}
If we assume that $\theta$ is {\it smooth} in the sense that it is continuous and piecewise continuously differentiable, then it admits a Fourier series representation
\begin{equation}\label{Fseries}
\theta(x)=\sum_{k\in\ganz}c_k\,e^{ik\tilde cx}\quad\text{ with }\quad\tilde c=\frac{2\pi\alpha}{\log c}.
\end{equation}
In this case the function $\gamma$ appearing in Lemma \ref{gammaadm} is given by the modified Fourier series
\begin{equation}\label{modFseries}
\gamma(x)=\sum_{k\in\ganz}c_k\,\Gamma(ik\tilde c-\alpha+1)\,e^{ik\tilde cx}\quad\text{ for }x\in\rr
\end{equation}
which can be seen as follows. By Theorem 3.1 in \cite{KLM} the coefficients of this series appear in a representation of the log-characteristic function
\begin{equation}\label{serieslogchar}
\psi(x)=-\sum_{k\in\ganz}c_k\,\Gamma(ik\tilde c-\alpha+1)\,(-ix)^{\alpha-ik\tilde c}\quad\text{ for }x\in\rr
\end{equation}
and the relation $\gamma(x)=e^{\alpha x}\,\widetilde\psi(e^{-x})=-e^{\alpha x}\psi(ie^{-x})$ easily shows \eqref{modFseries}.
\end{remark}
A natural question which arises is if for two admissable functions $\theta_1,\,\theta_2$ with respect to the parameters $\alpha_1\in(0,1)$ and $c_1>1$, respectively $\alpha_2\in(0,1)$ and $c_2>1$, the composition of the corresponding semi-fractional derivatives given by \eqref{semifd} can again be a semi-fractional derivative of order $\alpha:=\alpha_1+\alpha_2$. We concentrate on the easiest case when $\alpha\in(0,1)$ and $\theta_1,\,\theta_2$ have the same periodicity, i.e.\ $c_1^{1/\alpha_1}=c_2^{1/\alpha_2}$. In view of \eqref{semifdF} for the corresponding log-characteristic functions we need to show that 
\begin{equation}\label{sglogchar}
\psi_1(x)\cdot\psi_2(x)=-\psi(x)\quad\text{ for all }x\in\rr,
\end{equation}
where $\psi$ is the log-characteristic function of a semistable distribution corresponding to an admissable function $\theta$ with respect to the parameters $\alpha\in(0,1)$ and $c:=c_1^{\alpha/\alpha_1}=c_2^{\alpha/\alpha_2} >1$. Since \eqref{sglogchar} is equivalent to $\widetilde\psi_1\cdot\widetilde\psi_2=\widetilde\psi$ for the corresponding log-Laplace exponents and we have self-similarity
$$\widetilde\psi(c^{1/\alpha}x)=\widetilde\psi_1(c_1^{1/\alpha_1}x)\cdot\widetilde\psi_2(c_2^{1/\alpha_2}x)=c_1c_2\cdot\widetilde\psi_1(x)\widetilde\psi_2(x)=c^{\alpha_1/\alpha}c^{\alpha_2/\alpha}\cdot\widetilde\psi(x)=c\cdot\widetilde\psi(x),$$
in view of Lemma \ref{gammaadm} this is equivalent to require that $\widetilde\psi$ is a Bernstein function.
\begin{cor}\label{sfdsemigroup}
$ \widetilde\psi_1\cdot\widetilde\psi_2$ is a Bernstein function iff there exists an admissable function $\theta$ with respect to the parameters $\alpha=\alpha_1+\alpha_2\in(0,1)$ and $c=c_1^{\alpha/\alpha_1}=c_2^{\alpha/\alpha_2}>1$ such that
\begin{equation}\label{sfds3}
\frac{\partial^{\alpha_2}}{\partial_{c_2,\theta_2}x^{\alpha_2}}\,\frac{\partial^{\alpha_1}}{\partial_{c_1,\theta_1}x^{\alpha_1}}\,f=\frac{\partial^{\alpha}}{\partial_{c,\theta}x^{\alpha}}\,f
\end{equation}
for suitable functions $f\in W^{2,1}(\rr)$ with $\frac{\partial^{\alpha_1}}{\partial_{c_1,\theta_1}x^{\alpha_1}}\,f$ belonging to the domain of $\frac{\partial^{\alpha_2}}{\partial_{c_2,\theta_2}x^{\alpha_2}}$.
\end{cor}
\begin{remark}
In case $\theta_1,\,\theta_2$ are smooth admissable functions with Fourier series representations as in \eqref{Fseries} 
$$\theta_1(x)=\sum_{k\in\ganz}c_{k,1}\,e^{ik\tilde cx}\quad\text{ and }\quad\theta_2(x)=\sum_{\ell\in\ganz}c_{\ell,2}\,e^{i\ell\tilde cx},$$
where $\tilde c=\frac{2\pi\alpha_1}{\log c_1}=\frac{2\pi\alpha_2}{\log c_2}=\frac{2\pi\alpha}{\log c}$, then by \eqref{serieslogchar} and the Cauchy product rule we easily get
\begin{equation}\label{logLprod}
\widetilde\psi(x)=\widetilde\psi_1(x)\cdot\widetilde\psi_2(x)=\sum_{m\in\ganz}d_m\,\Gamma(im\tilde c-(\alpha_1+\alpha_2)+1)\,x^{\alpha_1+\alpha_2-im\tilde c},
\end{equation}
where
$$d_m=\sum_{\ell\in\ganz}c_{m-\ell,1}c_{\ell,2}\,\frac{\Gamma(i(m-\ell)\tilde c-\alpha_1+1)\,\Gamma(i\ell\tilde c-\alpha_2+1)}{\Gamma(im\tilde c-(\alpha_1+\alpha_2)+1)}.$$
Uniqueness of the Fourier coefficients then gives us the representation
\begin{equation}\label{thetadef}
\theta(x)=\sum_{m\in\ganz}d_m\,e^{im\tilde cx}
\end{equation}
for the admissable function $\theta$ from Corollary \ref{sfdsemigroup}, provided that \eqref{logLprod} is a Bernstein function.
\end{remark}

Finally, we want to show that the semi-fractional derivative of order $\alpha\in(0,1)$ can be seen as a special case of a generalized fractional derivative introduced in \cite{Koc}. Starting with \eqref{semifd} for a smooth admissable function $\theta$, integration by parts shows that 
$$\frac{\partial^\alpha}{\partial_{c,\theta}x^\alpha}\,f(x)=\int_0^\infty f'(x-y)\,y^{-\alpha}\theta(\log y)\,dy$$
as laid out in \cite{KLM}. Introducing the kernel function $k(y):=y^{-\alpha}\theta(\log y)$ and restricting considerations to functions with support on the positiv real line, this can be interpreted as a semi-fractional derivative of Caputo type
$$\frac{\partial^\alpha}{\partial_{c,\theta}x^\alpha}\,f(x)=\int_0^x f'(x-y)\,k(y)\,dy=(k\ast f')(x).$$
On the other hand, interchanging the order of integration and differentiation gives a semi-fractional derivative of Riemann-Liouville type
$$\mathbb D_{c,\theta}^\alpha f(x)=\frac{d}{dx}\int_0^x f(x-y)\,k(y)\,dy=(k\ast f)'(x)$$
which is also called of convolution type in \cite{Toa}. The relationship between these forms is given by the formula
$$\frac{\partial^\alpha}{\partial_{c,\theta}x^\alpha}\,f(x)=\mathbb D_{c,\theta}^\alpha f(x)-f(0)k(x)=:\mathbb D_{(k)}f(x)$$
which can be derived as (2.33) in \cite{MMMSik} and for more general kernel functions  $\mathbb D_{(k)}f$ is called a generalized fractional derivative in \cite{Koc}. For further approaches into this direction see \cite{Asc,KKdS,LRdS,PatSra,SMC,Toa}. Of particular interest are non-negative locally integrable kernel functions $k$ such that the operator $\mathbb D_{(k)}$ possesses a right inverse $\mathbb I_{(k)}$ such that $\mathbb D_{(k)}\mathbb I_{(k)}f=f$. Using the theory of complete Bernstein functions and the relationship to the Stieltjes class, it is shown in \cite{Koc} that this is possible with
$$\mathbb I_{(k)}f(x)=\int_0^x f(y)\,k^\ast(x-y)\,dy$$
for locally bounded measurable functions $f$ if $(k,k^\ast)$ forms a Sonine pair of kernels, i.e.\ $k\ast k^\ast\equiv 1$; cf.\ also \cite{Sonine,SamCar,Wick}. In this case $\mathbb I_{(k)}f$ is called a generalized fractional integral of order $\alpha$ and it also holds that $\mathbb I_{(k)}\mathbb D_{(k)}f(x)=f(x)-f(0)$ for absolutely continuous functions $f$. As shown in \cite{Han}, necessarily the kernel function $k$ must have an integrable singularity at $0$.
We will now show that a Sonine kernel $k^\ast$ may exist for our specific kernel function $k(y)=y^{-\alpha}\theta(\log y)$ working in the more general framework of Bernstein functions and their relation to completely monotone functions. Hence we aim to extend the list of specific kernel functions given in section 6 of \cite{LRdS} by a new example. Therefore we have to relax the definition of a Sonine pair in the following sense.
\begin{defi}
Given a non-negative, locally bounded and measurable function $k$ on $(0,\infty)$ and a sigma-finite Borel measure $\rho$ on $(0,\infty)$ we say that $(k,\rho)$ forms a {\it generalized Sonine pair} if $k\ast\rho\equiv 1$.
\end{defi}
We know that $\widetilde\psi(x)=\int_0^\infty(1-e^{-xt})\,d\phi(t)$ with $\phi$ as in \eqref{tailLm+} is a Bernstein function and thus $G(x)=\frac1{x}\,\widetilde\psi(x)$ is completely monotone by Corollary 3.8 (iv) in \cite{SSV}. Integration by parts yields the Laplace transform
\begin{equation}\label{LTG}
G(x)=\int_0^\infty\tfrac1{x}(1-e^{-xt})\,d\phi(t)=\int_0^\infty e^{-xt}k(t)\,dt=\widetilde k(x)\quad\text{ for }x>0
\end{equation}
of our kernel function $k(y)=y^{-\alpha}\theta(\log y)$. By Theorem 3.7 in \cite{SSV} $G^\ast(x)=1/\widetilde\psi(x)$ is completely monotone since it is the composition of a Bernstein function and the completely monotone function $x\mapsto\frac1{x}$. Hence there exists a Borel measure $\rho$ on $[0,\infty)$ such that $G^\ast(x)=\int_0^\infty e^{-xt}\,d\rho(t)$ which serves as the generalized Sonine partner of $k$.
\begin{lemma}\label{Sonine}
For a smooth admissable function $\theta$ with respect to $\alpha\in(0,1)$ and $c>1$ let $k(y)=y^{-\alpha}\theta(\log y)$ for $y>0$ and $\rho$ be the Borel measure on $(0,\infty)$ with Laplace transform $\widetilde\rho(x)=G^\ast(x)$ as above. Then $(k,\rho)$ forms a generalized Sonine pair.
\end{lemma}
\begin{proof}
We easily calculate the Laplace transform of $(k\ast \rho)(t)=\int_0^tk(t-x)\,d\rho(x)$ as
$$\widetilde k(x)\cdot \widetilde\rho(x)=G(x)\cdot G^\ast(x)=\frac{\widetilde\psi(x)}{x}\cdot\frac{1}{\widetilde\psi(x)}=\frac1x$$
for $x>0$, which is the Laplace transform of $1_{(0,\infty)}$.
\end{proof}
By virtue of Lemma \ref{Sonine} we may interpret 
$$\mathbb I_{(k)}f(x)=\int_0^x f(x-y)\,d\rho(y)=(f\ast\rho)(x)$$
as a {\it semi-fractional integral} of order $\alpha$ for locally bounded and measurable functions $f$. Since $\mathbb I_{(k)}f(0)=0$, we get
\begin{align*}
\mathbb D_{(k)}\mathbb I_{(k)}f(x) & =\frac{d}{dx}\int_0^x\mathbb I_{(k)}f(x-y)\,k(y)\,dy=\frac{d}{dx}(\mathbb I_{(k)}f\ast k)(x)\\
& =\frac{d}{dx}((f\ast\rho)\ast k)(x)=\frac{d}{dx}((k\ast\rho)\ast f)(x)=\frac{d}{dx}(1_{(0,\infty)}\ast f)(x)\\
& =\frac{d}{dx}\int_0^xf(t)\,dt=f(x)
\end{align*}
and for absolutely continuous functions $f$ with density $f'$ we get
\begin{align*}
\mathbb I_{(k)}\mathbb D_{(k)}f(x) & =\int_0^x\mathbb D_{(k)}f(x-y))\,d\rho(y)=\int_0^x (k\ast f')(x-y)\,d\rho(y)\\
& =((k\ast f')\ast \rho)(x)=((k\ast\rho)\ast f')(x)=(1_{(0,\infty)}\ast f')(x)\\
& =\int_0^xf'(t)\,dt=f(x)-f(0).
\end{align*}
The semi-fractional integral is also determined by a self-similar Bernstein function.
\begin{lemma}
The primitive $G_I^\ast(x):=\int_0^xG^\ast(y)\,dy=\int_0^x\widetilde\rho(y)\,dy$ is a self-similar Bernstein function with respect to $1-\alpha\in(0,1)$ and $d:=c^{\frac{1-\alpha}{\alpha}}>1$.
\end{lemma}
\begin{proof}
From Lemma \ref{gammaadm} we get the scaling relation
$$G^\ast(c^{1/\alpha}x)=\frac1{\widetilde\psi(c^{1/\alpha}x)}=\frac1{c\cdot\widetilde\psi(x)}=c^{-1}G^\ast(x)\quad \text{ for all }x>0$$
which implies $(c^{1/\alpha}\rho)=c^{-1}\cdot\rho$ for the measure $\rho$ in Lemma \ref{Sonine}. In particular it follows that $\rho(\{0\})=0$ and by the Fubini-Tonelli theorem we get
\begin{align*}
G_I^\ast(x)= & \int_0^xG^\ast(y)\,dy=\int_0^x\int_{0+}^\infty e^{-yt}\,d\rho(t)\,dy\\
= & \int_{0+}^\infty \frac{1-e^{-xt}}{t}\,d\rho(t)=\int_{0}^\infty (1-e^{-xt})\,d\mu(t),
\end{align*}
where the Borel measure $\mu$ on $(0,\infty)$ is given by $d\mu(t):=\frac1t d\rho(t)$ and integrates $\min\{1,t\}$ as shown in the proof of Theorem 3.2 in \cite{SSV}; cf.\ Proposition 3.5 in \cite{SSV}. Thus the primitive $G_I^\ast(x)$ is a Bernstein function and the scaling relation gives us
\begin{align*}
G_I^\ast(c^{1/\alpha}x):= & \int_{0}^\infty \left(1-e^{-xc^{1/\alpha}t}\right)\,d\mu(t)=c^{1/\alpha}\int_{0+}^\infty \frac{1-e^{-xc^{1/\alpha}t}}{c^{1/\alpha}\,t}\,d\rho(t)\\
= & c^{1/\alpha}\int_{0+}^\infty \frac{1-e^{-xt}}{t}\,d(c^{1/\alpha}\rho)(t)= c^{\frac1{\alpha}-1}\int_{0}^\infty (1-e^{-xt})\,d\mu(t)=c^{\frac{1-\alpha}{\alpha}}G_I^\ast(x).
\end{align*}
Now let $d=c^{\frac{1-\alpha}{\alpha}}>1$ then $d^{\frac1{1-\alpha}}=c^{1/\alpha}$ and we have $G_I^\ast(d^{\frac1{1-\alpha}}x)=d\cdot G_I^\ast(x)$ for all $x>0$. Hence $G_I^\ast$ is a self-similar Bernstein function with respect to $1-\alpha\in(0,1)$ and $d>1$. 
\end{proof}
\begin{remark}
By Lemma \ref{gammaadm} the L\'evy measure $\mu$ corresponding to $G_I^\ast$ is given by $\mu(t,\infty)=t^{\alpha-1}\sigma(\log t)$ for an admissable function $\sigma$ with respect to the parameters $1-\alpha\in(0,1)$ and $d>1$. If this admissable function $\sigma$ is smooth, we get a Sonine pair $(k,k^\ast)$ in the original sense as follows. Integration by parts yields
$$G_I^\ast(x)=x\int_{0}^\infty \tfrac1x(1-e^{-xt})\,d\mu(t)=x\int_{0}^\infty e^{-xt}t^{\alpha-1}\sigma(\log t)\,dt$$
and hence we get as the derivative
\begin{align*}
G^\ast(x) & =\int_{0}^\infty e^{-xt}t^{\alpha-1}\sigma(\log t)\,dt-x\int_{0}^\infty e^{-xt}t^{\alpha}\sigma(\log t)\,dt\\
& =\int_{0}^\infty e^{-xt}\left(t^{\alpha-1}\sigma(\log t)-\tfrac{d}{dt}\left(t^{\alpha}\sigma(\log t)\right)\right)\,dt\\
& =\int_{0}^\infty e^{-xt}\,t^{\alpha-1}\left((1-\alpha)\sigma(\log t)-\sigma'(\log t)\right)\,dt\\
& =:\int_{0}^\infty e^{-xt}\,k^\ast(t)\,dt,
\end{align*}
where the kernel $k^\ast$ is non-negative by Lemma A.1. To show that $(k,k^\ast)$ is a Sonine pair simply calculate the Laplace transform as in the proof of Lemma \ref{Sonine}. Note that in general we cannot expect $k^\ast$ to be completely monotone as in the approach of \cite{Koc} with complete Bernstein functions. Nor can we expect that the admissible function $\sigma$ is smooth in general.
\end{remark}

\section{Discrete approximation of the generator}

Recall that by \eqref{semifd} the semi-fractional derivative operator of order $\alpha\in(0,1)$ is given by the negative generator of the continuous convolution semigroup $(\nu^{\ast t})_{t\geq0}$, where $\nu$ is the semistable distribution with log-characteristic function \eqref{logchar+}. If the tail of the L\'evy measure is given by $\phi(t,\infty)=\Gamma(1-\alpha)\,t^{-\alpha}$, i.e.\ the admissable function $\theta\equiv\Gamma(1-\alpha)$ is constant, it is well known that the semi-fractional derivative coincides with the ordinary Riemann-Liouvile fractional derivative of order $\alpha\in(0,1)$ and can be approximated by means of the Gr\"unwald-Letnikov formula
\begin{equation}\label{GL}
\frac{\partial^{\alpha}}{\partial x^{\alpha}}\,f(x)=\lim_{h\downarrow0}h^{-\alpha}\sum_{j=0}^{\infty}{\alpha\choose j}(-1)^{j}f(x-jh)
\end{equation}
for functions $f$ in the Sobolev space $W^{2,1}(\rr)$; see section 2.1 in \cite{MMMSik} for details. The coefficients in the Gr\"unwald-Letnikov approximation formula appear in a discrete distribution on the positive integers called {\it Sibuya distribution} which first appeared in \cite{Sib}.  A discrete random variable $X_{\alpha}$ on $\nat$ is Sibuya distributed with parameter $\alpha\in(0,1)$ if
\begin{equation}\label{sibuya}
\mathbb P(X_{\alpha}=j)=(-1)^{j-1}{\alpha\choose j}=(-1)^{j-1}\frac{\alpha(\alpha-1)\cdots(\alpha-j+1)}{j!}\quad\text{ for }j\in\nat.
\end{equation}
For further details and extensions of the Sibuya distribution we refer to \cite{Bou,ChrSch,KozPod} and the literature mentioned therein. Using \eqref{sibuya} we may rewrite \eqref{GL} as
\begin{align*}
\frac{\partial^{\alpha}}{\partial x^{\alpha}}\,f(x) & =\lim_{h\downarrow0}h^{-\alpha}\sum_{j=0}^{\infty}{\alpha\choose j}(-1)^{j}f(x-jh)\\
& =\lim_{h\downarrow0}h^{-\alpha}\left(f(x)-\sum_{j=1}^{\infty}\mathbb P(X_{\alpha}=j)f(x-jh)\right)\\
& =\lim_{h\downarrow0}h^{-\alpha}\left(f(x)-\int_{\rr}f(x-hy)\,d\mathbb P_{X_{\alpha}}(y)\right)\\
& =\lim_{h\downarrow0}h^{-\alpha}\left(f\ast\left(\varepsilon_{0}-\mathbb P_{hX_{\alpha}}\right)\right)(x)
\end{align*}
which shows that the Gr\"unwald-Letnikov formula is in fact a discrete approximation of the generator. For further relations of the Sibuya distribution to fractional diffusion equations see \cite{NHA,PPR}.

Our aim is to generalize this formula for semi-fractional derivatives by means of the corresponding self-similar Bernstein functions. Note that for the Bernstein function $\widetilde\psi(x)=x^{\alpha}$ the nominator on the right-hand side of \eqref{sibuya} is given by $\widetilde\psi^{(j)}(1)$ and hence we may define a {\it semi-fractional Sibuya distribution} in the following way.
\begin{defi}
Given an admissable function $\theta$ with respect to $\alpha\in(0,1)$ and $c>1$, a semi-fractional Sibuya distributed random variable $X_{\theta}$ on $\nat$ is given by
\begin{equation}\label{semisib}
\mathbb P(X_{\theta}=j)=\frac{(-1)^{j-1}}{j!}\,\frac{\widetilde\psi^{(j)}(1)}{\widetilde\psi(1)}\quad\text{ for }j\in\nat,
\end{equation}
where $\widetilde\psi$ is the corresponding self-similar Bernstein function.
\end{defi}
Clearly, the expression in \eqref{semisib} is non-negative by \eqref{Bern1}. Moreover, by monotone convergence and a Taylor series approach justified by Proposition 3.6 in \cite{SSV} we get
\begin{align*}
\frac{(-1)^{j-1}}{j!}\,\frac{\widetilde\psi^{(j)}(1)}{\widetilde\psi(1)} & =\frac1{\widetilde\psi(1)}\sum_{j=1}^{\infty}\frac{\widetilde\psi^{(j)}(1)}{j!}\lim_{\varepsilon\downarrow0}(\varepsilon-1)^{j-1}\\
& =\frac1{\widetilde\psi(1)}\lim_{\varepsilon\downarrow0}\frac1{\varepsilon-1}\left(\sum_{j=0}^{\infty}\frac{\widetilde\psi^{(j)}(1)}{j!}(\varepsilon-1)^{j}-\widetilde\psi(1)\right)\\
& =\frac1{\widetilde\psi(1)}\lim_{\varepsilon\downarrow0}\frac1{\varepsilon-1}\left(\widetilde\psi(\varepsilon)-\widetilde\psi(1)\right)=\frac{\widetilde\psi(1)-\widetilde\psi(0)}{\widetilde\psi(1)}=1.
\end{align*}
This shows that indeed \eqref{semisib} defines a proper distribution on $\nat$ with pgf
\begin{align*}
G(z) & =\sum_{j=1}^{\infty}\mathbb P(X_{\theta}=j)\, z^{j} = -\,\frac1{\widetilde\psi(1)}\sum_{j=1}^{\infty}\frac{\widetilde\psi^{(j)}(1)}{j!}\,(-z)^{j}\\
& =1-\frac1{\widetilde\psi(1)}\sum_{j=0}^{\infty}\frac{\widetilde\psi^{(j)}(1)}{j!}\,((1-z)-1)^{j}=1-\frac{\widetilde\psi(1-z)}{\widetilde\psi(1)}
\end{align*}
for $|z|\leq1$. Note that in the above arguments self-similarity of the Bernstein function is not needed. Hence \eqref{semisib} defines a proper distribution on $\nat$ for every Bernstein function $\widetilde\psi$ with $\widetilde\psi(0)=0$.

Now let $\theta$ be a smooth admissable function having Fourier series representation
$$\theta(x)=\sum_{k=-\infty}^{\infty}c_{k} e^{ik\tilde c x}\quad\text{ with }\quad\tilde c=\tfrac{2\pi\alpha}{\log c}.$$
Then by Theorem 3.1 in \cite{KLM} the corresponding self-similar Bernstein function can be expressed in terms of the log-characteristic function $\psi$ as
$$\widetilde\psi(x)=-\psi(ix)=\sum_{k=-\infty}^{\infty}\omega_{k} x^{\alpha-ik\tilde c}\quad\text{ with }\quad\omega_{k}=c_{k}\Gamma(ik\tilde c-\alpha+1)$$
and hece for $j\in\nat_{0}$ we have
$$\frac{\widetilde\psi^{(j)}(1)}{j!}=\sum_{k=-\infty}^{\infty}\omega_{k}{\alpha-ik\tilde c\choose j}.$$
Note that these coefficients appear in a Gr\"unwald-Letnikov type approximation of the semi-fractional derivative given in Theorem 4.1 of \cite{KLM} which enables us to prove the following approximation formula.
\begin{theorem}
Let $\theta$ be a smooth admissable function with respect to $\alpha\in(0,1)$ and $c>1$ with corresponding self-similar Bernstein function $\widetilde\psi$ and semi-fractional Sibuya distributed random variable $X_{\theta}$ given by \eqref{semisib}. Then for $f\in W^{2,1}(\rr)$ the semi-fractional derivative can be approximated along the subsequence $h_{m}=c^{-m/\alpha}$ by
$$\frac{\partial^{\alpha}}{\partial_{c,\theta} x^{\alpha}}\,f(x)=\lim_{m\to\infty}\widetilde\psi(h_{m}^{-1})\left(f\ast\left(\varepsilon_{0}-\mathbb P_{h_{m }X_{\theta}}\right)\right)(x).$$
\end{theorem}
\begin{proof}
First note that $h_{m}^{ik\tilde c}=c^{-imk\tilde c/\alpha}=e^{-2\pi imk}=1$ and by self-similarity we have $\widetilde\psi(h_{m}^{-1})=h_{m}^{-\alpha}\widetilde\psi(1)$ for all $m\in\nat$. Hence we get
\begin{align*}
& \widetilde\psi(h_{m}^{-1})\left(f\ast\left(\varepsilon_{0}-\mathbb P_{h_{m }X_{\theta}}\right)\right)(x) =h_{m}^{-\alpha}\widetilde\psi(1)\left(f(x)-\int_{\rr}f(x-h_{m}y)\,d\mathbb P_{X_{\theta}}(y)\right)\\
& \quad =h_{m}^{-\alpha}\widetilde\psi(1)\left(f(x)-\sum_{j=1}^{\infty}\frac{(-1)^{j-1}}{j!}\,\frac{\widetilde\psi^{(j)}(1)}{\widetilde\psi(1)}\,f(x-jh_{m})\right)\\
& \quad =h_{m}^{-\alpha}\left(\sum_{k=-\infty}^{\infty}\omega_{k}f(x)+\sum_{j=1}^{\infty}(-1)^{j}\sum_{k=-\infty}^{\infty}\omega_{k}{\alpha-ik\tilde c\choose j}f(x-jh_{m})\right)\\
& \quad =h_{m}^{-\alpha}\sum_{j=0}^{\infty}(-1)^{j}\sum_{k=-\infty}^{\infty}\omega_{k}h_{m}^{ik\tilde c} {\alpha-ik\tilde c\choose j}f(x-jh_{m})\to\frac{\partial^{\alpha}}{\partial_{c,\theta} x^{\alpha}}\,f(x),
\end{align*}
where the last convergence follows from Theorem 4.1 in \cite{KLM}.
\end{proof}

\section{Space-time duality for semi-fractional diffusions}

We first give a sufficient condition for an inverse function to be a Bernstein function.
\begin{lemma}\label{invBernstein}
Let $f:(0,\infty)\to(0,\infty)$ be a $C^\infty(0,\infty)$-function such that $f'$ is a Bernstein function and $f^{(n)}(x)\not=0$ for all $x>0$ and $n\in\nat$. Then its inverse $f^{-1}$ is a Bernstein function with $(f^{-1})^{(n)}(x)\not=0$ for all $x>0$ and $n\in\nat$. 
\end{lemma}
\begin{proof}
We know that $f^{-1}(x)>0$ and $(f^{-1})'(x)=\frac{1}{f'(f^{-1}(x))}>0$ for all $x>0$. Moreover, as in Remark A.3 of the Appendix $(f^{-1})''(x)=-\,\frac{f''(f^{-1}(x))}{(f'(f^{-1}(x)))^3}<0$ for all $x>0$. For $n\geq3$ we inductively use the formula for $(f^{-1})^{(n)}$ given in Lemma A.2 of the Appendix. By induction we have
$$\sign\left(\prod_{j=1}^{n-1}\left(\frac{(f^{-1})^{(j)}(x)}{j!}\right)^{k_j}\right)=\prod_{j=1}^{n-1}(-1)^{(j-1)k_j}=(-1)^{\sum_{j=1}^{n-1}(j-1)k_j}$$
and $\sign(f^{(k_1+\cdots+k_{n-1})}(f^{-1}(x)))=(-1)^{k_1+\cdots+k_{n-1}}$ since $k_1+\cdots+k_{n-1}\geq 2$ for $n\geq3$ as $k_1+2k_2\cdots+(n-1)k_{n-1}=n$. This shows that
$$\sign\left(f^{(k_1+\cdots+k_{n-1})}(f^{-1}(x))\prod_{j=1}^{n-1}\left(\frac{(f^{-1})^{(j)}(x)}{j!}\right)^{k_j}\right)=(-1)^{\sum_{j=1}^{n-1}jk_j}=(-1)^n$$
for every summand of the formula in Lemma A.2. Thus $\sign((f^{-1})^{(n)}(x))=(-1)^{n-1}$ for all $x>0$ and $n\geq3$ showing that $f^{-1}$ is a Bernstein function.
\end{proof}
We want to apply Lemma \ref{invBernstein} to the function
\begin{equation}\label{fctf}
\zeta(x)=\psi(-ix)=\int_{-\infty}^0\left(e^{xy}-1-xy\right)\,d\phi(y)>0\quad\text{ for }x>0,
\end{equation}
where $\phi$ is the semistable L\'evy measure for $\alpha\in(1,2)$ from \eqref{tailLm-} concentrated on the negative axis and $\psi$ is the corresponding log-characteristic function from \eqref{logchar-}. Note that $\zeta(x)\to\infty$ as $x\to\infty$ due to the tail behavior in \eqref{tailLm-} and by dominated convergence we have
$$\zeta'(x)=\int_{-\infty}^0y\left(e^{xy}-1\right)\,d\phi(y)=\int_0^\infty\left(1-e^{-xy}\right)y\,d\phi(-y).$$
Since $\phi$ integrates $\min\{1,y^2\}$ and fulfills \eqref{tailLm-}, the measure $\mu$ with $d\mu(y)=y\,d\phi(-y)$ integrates $\min\{1,y\}$ and thus $\zeta'$ is a Bernstein function. This also follows from the higher order derivatives
$$\zeta^{(n)}(x)=\int_{-\infty}^0y^ne^{xy}\,d\phi(y)\begin{cases} >0 & \text{ if $n\geq2$ is even,}\\ <0 & \text{ if $n\geq3$ is odd,}\end{cases}$$
which additionally shows that these derivatives do not vanish. Thus from Lemma \ref{invBernstein} we conclude that the inverse $\zeta^{-1}$ is a Bernstein function. This will be the key to solve an open problem concerning the following result on space-time duality from \cite{KL}.

According to \cite[Example 28.2]{Sato}, the semistable L\'evy process $(X_t)_{t\geq0}$ with $P\{X_t\in A\}=\nu^{\ast t}(A)$ for Borel sets $A\in\mathcal B(\rr)$, where $\nu$ is the semistable distribution with log-characteristic function $\psi$ from \eqref{logchar-}, possesses $C^\infty(\rr)$-densities $x\mapsto p(x,t)$ for every $t>0$ with $P\{X_t\in A\}=\int_Ap(x,t)\,dx$ and for $t=0$ we may write $p(x,0)=\delta(x)$ corresponding to $X_{0}=0$ almost surely. It is shown in \cite{KLM} that these densities are the point source solution to the semi-fractional diffusion equation
\begin{equation}\label{sfdeq}
\frac{\partial^\alpha}{\partial_{c,\theta}(-x)^\alpha}\,p(x,t)=\frac{\partial}{\partial t}\,p(x,t)
\end{equation}
with the negative semi-fractional derivative of order $\alpha$ from \eqref{negsemifd} acting on the space variable. Since the semi-fractional derivative is a non-local operator, this equation is hard to interpret from a physical point of view, whereas non-locality in time may correspond to long memory effects \cite{Hil}. As a generalization of a space-time duality result for fractional diffusions \cite{BMN, KelMMM} based on a corresponding result for stable densities by Zolotarev \cite{Zol, Zolbook} it was shown in \cite{KL} that space-time duality may also hold for semi-fractional diffusions. Theorem 3.3 in \cite{KL} states that for $x>0$ and $t>0$ we have $p(x,t)=\alpha^{-1}h(x,t)$, where $h(x,t)$ is the point source solution to the semi-fractional differential equation
\begin{equation}\label{sfdeqtime}
\frac{\partial^{1/\alpha}}{\partial_{d,\tau}t^{1/\alpha}}\,h(x,t)+\frac{\partial}{\partial x}\,h(x,t)=t^{-1/\alpha}\varrho(\log t)\,\delta(x)
\end{equation}
with a semi-fractional derivative of order $1/\alpha$ acting on the time variable, provided that $\tau$ and $\varrho$ are admissable functions with respect to $1/\alpha\in(\frac12,1)$ and $d=c^{1/\alpha}>1$ which remains an open problem in \cite{KL}. We will now show that indeed $\tau$ is admissable and $\varrho(x)=-\alpha\tau'(x)$, provided that $\tau$ is smooth. Note that in general $\varrho$ will not be admissable as conjectured in \cite{KL} but we will justify the inhomogenity in \eqref{sfdeqtime} by different arguments. In \cite{KL} the function $\tau$ appears in the following way. The above inverse $\zeta^{-1}$ is called $\xi$ in \cite{KL} and its existence is shown in Lemma 4.1 of \cite{KL}. It was further shown in Lemma 4.2 of \cite{KL} that $\xi(t)=t^{1/\alpha}g(\log t)$ for a continuously differentiable and $\log(c)$-periodic function $g$. Since we now know that $\xi=\zeta^{-1}$ is a Bernstein function, in fact $g$ is a $C^\infty(\rr)$-function. Since $g$ is a smooth $\log(c)$-periodic function, it is representable by its Fourier series
\begin{equation}\label{FSg}
g(x)=\sum_{n\in\ganz}d_ne^{-in\tilde dx}\quad\text{ with }\tilde d=\frac{2\pi}{\log c}=\frac{2\pi\frac1{\alpha}}{\log d}\text{ for }d=c^{1/\alpha},
\end{equation}
where by Lemma 1 in \S12 of \cite{Arn} we have $|d_{n}|\leq C\cdot e^{-\frac{\pi}{2}\,|n|\tilde d}$ for some $C>0$ and all $n\in\ganz$. If we require a little more quality, namely that the Fourier coefficients even decay as 
\begin{equation}\label{taucond}
|d_{n}|\leq C\cdot e^{-\frac{\pi}{2}\,|n|\tilde d}|n|^{-\frac32-\frac1{\alpha}-\varepsilon}\quad\text{ for some $\varepsilon>0$ and all $n\in\ganz\setminus\{0\}$},
\end{equation}
then we can define $\tau$ by the Fourier series
\begin{equation}\label{FStau}
\tau(x)=\sum_{n\in\ganz}\frac{d_n}{\Gamma(in\tilde d-\frac1{\alpha}+1)}\,e^{-in\tilde dx}.
\end{equation}
\begin{lemma}\label{tauadm}
If \eqref{taucond} holds, then the function $\tau$ in \eqref{FStau} is well-defined and a smooth admissable function with respect to $1/\alpha\in(\frac12,1)$ and $d=c^{1/\alpha}$. Moreover, $\zeta^{-1}$ is a self-similar Bernstein function with respect to the same parameters.
\end{lemma}
\begin{proof}
As shown above $\xi=\zeta^{-1}$ is a Bernstein function and with $d=c^{1/\alpha}$ and $\xi(t)=t^{1/\alpha}g(\log t)$ for a $\log(c)$-periodic function $g$ we get
$$d\cdot\xi(t)=(ct)^{1/\alpha}g(\log(ct))=\xi(ct)=\xi(d^\alpha t)$$
showing that $\xi$ is a self-similar Bernstein function with respect to $1/\alpha\in(\frac12,1)$ and $d=c^{1/\alpha}$. By Lemma \ref{gammaadm} we have
\begin{equation}\label{xiBf}
\xi(x)=\int_0^\infty\left(1-e^{-xy}\right)\,d\mu(y),
\end{equation}
where $\mu$ is a semistable L\'evy measure with $\mu(t,\infty)=t^{-1/\alpha}\tau(\log t)$ for an admissable function $\tau$  with respect to $1/\alpha\in(\frac12,1)$ and $d=c^{1/\alpha}$. It remains to show that $\tau$ is indeed the function we are looking for. Let us assume for a moment that $\tau$ is smooth and thus admits the Fourier series
$$\tau(x)=\sum_{n\in\ganz}a_n\,e^{-in\tilde dx}\quad\text{ with }\tilde d=\frac{2\pi}{\log c}=\frac{2\pi\frac1{\alpha}}{\log d}.$$
In this case the function $\gamma$ appearing in Lemma \ref{gammaadm} fulfills
$$\gamma(-x)=e^{-\frac1{\alpha}\,x}\xi(e^x)=g(x)$$
and by \eqref{modFseries} has the Fourier series representation
$$g(x)=\gamma(-x)=\sum_{n\in\ganz}a_n\Gamma(in\tilde d-\tfrac1{\alpha}+1)\,e^{-in\tilde dx}.$$
A comparison with \eqref{FSg} and uniqueness of the Fourier coefficients shows that $a_n$ coincides with $d_n/\Gamma(in\tilde d-\tfrac1{\alpha}+1)$ and thus $\tau$ is indeed the function in \eqref{FStau}. Finally, we have to show that the series in \eqref{FStau} converges. Using the asymptotic behavior of the gamma function in Corollary 1.4.4 of \cite{AAR}, the Fourier coefficients fulfill
$$\left|\frac{d_n}{\Gamma(in\tilde d-\frac1{\alpha}+1)}\right|\leq K |d_{n}|\cdot|n|^{-\frac12+\frac1{\alpha}} e^{\frac{\pi}{2}\,|n|\tilde d}$$
for a constant $K>0$ and all $n\in\ganz\setminus\{0\}$. According to our assumption \eqref{taucond} we obtain
$$\left|\frac{d_n}{\Gamma(in\tilde d-\frac1{\alpha}+1)}\right|\leq KC|n|^{-2-\varepsilon}$$
for some $\varepsilon >0$ and all $n\in\ganz\setminus\{0\}$ showing that the series in \eqref{FStau} converges and the resulting function is continuously differentiable by Theorem 2.6 in \cite{Fol}.
\end{proof}
Admissability of $\tau$ in combination with Theorem 3.3 in \cite{KL} finally enables us to completely solve space-time duality for semi-fractional diffusions. The equation \eqref{sfdeqtime} is derived in \cite{KL} by Laplace inversion of the equation
\begin{equation}\label{sfdeqLT}
\xi(s)\widetilde h(x,s)-\frac1s\,\xi(s)\delta(x)+\frac{\partial}{\partial x}\,\widetilde h(x,s)=-\frac1s\,\frac{f(s)}{s+f(s)}\,\xi(s)\delta(x),
\end{equation}
where $f$ is defined in the proof of Theorem 3.1 in \cite{KL} as
\begin{equation}\label{deff}
f(s)=\frac1{\alpha}\,\xi(s)^\alpha m'(\log\xi(s))
\end{equation}
and $m$ is a $\log(c^{1/\alpha})$-periodic function given by 
\begin{equation}\label{defm}
\zeta(s)=x^\alpha m(\log x).
\end{equation}

\begin{theorem}\label{stduality}
Assume that \eqref{taucond} holds, $\tau$ is given as in Lemma \ref{tauadm} and define $\varrho(t)=-\alpha\tau'(t)$. Then the point source solutions $p(x,t)$ of the semi-fractional diffusion equation \eqref{sfdeq} of order $\alpha\in(1,2)$ in space and $h(x,t)$ of the semi-fractional equation \eqref{sfdeqtime} of order $1/\alpha\in(\frac12,1)$ in time are equivalent, i.e.\ $p(x,t)=\alpha^{-1}h(x,t)$ for all $x>0$ and $t>0$.
\end{theorem}
\begin{proof}
As shown in \cite{KL} it remains to carry out Laplace inversion of \eqref{sfdeqLT}. Clearly, $\frac{\partial}{\partial x}\,\widetilde h(x,s)$ is the Laplace transform of $\frac{\partial}{\partial x}\,h(x,t)$. Moreover, the first part on the left-hand side of \eqref{sfdeqLT} is the Laplace transform of the semi-fractional derivative in time $\frac{\partial^{1/\alpha}}{\partial_{d,\tau}t^{1/\alpha}}\,h(x,t)$ as argued in \cite{KL}. We may rewrite the right-hand side of \eqref{sfdeqLT} as follows. From \eqref{defm} we obtain
$$\zeta'(t)=\alpha\,x^{\alpha-1}m(\log x)+x^{\alpha-1}m'(\log x)$$ 
and thus we have by \eqref{defm} and $\zeta(\xi(s))=s$
$$\zeta'(\xi(s))=\alpha\,\frac1{\xi(s)}\,s+\frac1{\xi(s)}\,\xi(s)^{\alpha}m'(\log\xi(s))=\alpha\,\frac1{\xi(s)}\,s+\alpha\,\frac1{\xi(s)}\,f(s),$$
where the last equality follows from \eqref{deff}. This shows that 
$$f(s)=\frac1{\alpha}\,\zeta'(\xi(s))\,\xi(s)-s$$
and the right-hand side of \eqref{sfdeqLT} can be rewritten as 
\begin{equation}\label{CtoRL}
-\frac1s\,\frac{f(s)}{s+f(s)}\,\xi(s)\delta(x)=-\frac{\alpha}{s}\,\frac{f(s)}{\zeta'(\xi(s))}\,\delta(x)=-\frac1s\,\xi(s)\delta(x)+\alpha\,\frac1{\zeta'(\xi(s))}\,\delta(x).
\end{equation}
Note that the first part $-\frac1s\,\xi(s)\delta(x)$ also appears on the left-hand side of \eqref{sfdeqLT}; cf.\ Remark \ref{concl}. Moreover, since $\frac1s$ is the Laplace transform of $1_{(0,\infty)}(t)$, the function $\frac1s\,\xi(s)$ is the Riemann-Liouville semi-fractional derivative of $1_{(0,\infty)(t)}$ which is
$$\frac{d}{dt}\int_0^t 1_{(0,\infty)}(t-s)\,s^{-1/\alpha}\tau(\log(s))\,ds=t^{-1/\alpha}\tau(\log(t)).$$
Now from $\zeta(\xi(s))=s$ we get $\zeta'(\xi(s))\cdot \xi'(s)=1$ and hence it follows from \eqref{xiBf} and integration by parts
\begin{align*}
\frac1{\zeta'(\xi(s))} & = \xi'(s)=\int_0^\infty e^{-st}t\,d\mu(t)\\
& = \left[-e^{-st}t^{1-1/\alpha}\tau(\log t)\right]_{t=0}^\infty+\int_0^\infty\left(-s\,e^{-st}t+e^{-st}\right)t^{-1/\alpha}\tau(\log t)\,dt\\
& = -s\int_0^\infty e^{-st}t^{1-1/\alpha}\tau(\log t)\,dt+\int_0^\infty e^{-st}t^{-1/\alpha}\tau(\log t)\,dt
\end{align*}
which is the Laplace transform of 
\begin{align*}
& -\frac{d}{dt}\left(t^{1-1/\alpha}\tau(\log t)\right)+t^{-1/\alpha}\tau(\log t)\\
& \quad = -(1-\tfrac1{\alpha})t^{-1/\alpha}\tau(\log t)-t^{-1/\alpha}\tau'(\log t)+t^{-1/\alpha}\tau(\log t)\\
& \quad =t^{-1/\alpha}\left(\tfrac1{\alpha}\,\tau(\log t)-\tau'(\log t)\right).
\end{align*}
Putting things together, Laplace inversion of the right-hand side of \eqref{sfdeqLT} yields
$$\left(-t^{-1/\alpha}\tau(\log t)+\alpha t^{-1/\alpha}\left(\tfrac1{\alpha}\,\tau(\log t)-\tau'(\log t)\right)\right)\delta(x)=-\alpha t^{-1/\alpha}\tau'(\log t)\delta(x)$$
concluding the proof.
\end{proof}
\begin{remark}\label{concl}
Note that in the stable case we have that $\tau$ is constant and thus $\tau'\equiv0$. Hence we recover space-time duality for fractional diffusions in \cite{KelMMM} as a special case. Further note that with \eqref{CtoRL} the term $-\frac1s\,\xi(s)\delta(x)$ appears on both sides of \eqref{sfdeqLT} and can be cancelled. After cancellation the first part on the left-hand side of \eqref{sfdeqLT} is the Laplace transform of the Riemann-Liouville semi-fractional derivative of $h(x,t)$ and as in the proof of Theorem \ref{stduality} we may rewrite \eqref{sfdeqtime} as
$$\mathbb D_{d,\tau}^{1/\alpha}h(x,t)+\frac{\partial}{\partial x}\,h(x,t)=t^{-1/\alpha}\left(\tau(\log t)-\alpha\tau'(\log t) \right)\delta(x),$$
where the semi-fractional derivative acts on the time variable and now the inhomogeniety on the right-hand side is non-negative by Lemma A.1.
\end{remark}

\appendix
\section*{Appendix}
\renewcommand{\thesection}{A}
\setcounter{equation}{0}
\setcounter{theorem}{0}

We first show an elementary result connecting admissable functions with the representation in Corollary 7.4.4 of \cite{MMMHPS}.
\begin{lemma}
Let $\theta:\rr\to(0,\infty)$ be a periodic function and $\alpha>0$. Then the following statements are equivalent.
\begin{enumerate}
\item[(i)] $\theta(y+\delta)\leq e^{\alpha\delta}\theta(y)$ for all $y>0$ and $\delta>0$.
\item[(ii)] $\theta(y+\delta)\leq e^{\alpha\delta}\theta(y)$ for all $y\in\rr$ and $\delta\geq0$.
\item[(iii)] $\theta(y-\delta)\geq e^{-\alpha\delta}\theta(y)$ for all $y\in\rr$ and $\delta\geq0$.
\item[(iv)] The mapping $t\mapsto t^{-\alpha}\theta(\log t)$ is non-increasing for $t>0$.
\end{enumerate}
If $\theta$ is additionally differentiable, then each statement {\rm (i)--(iv)} is equivalent to
\begin{enumerate}
\item[(v)] $\theta'(y)\leq\alpha\,\theta(y)$ for all $y\in\rr$.
\end{enumerate}
\end{lemma}
\begin{proof}
{\rm ``(i)$\Rightarrow$(ii)'':} Note that {\rm (ii)} is trivially true for $\delta=0$. If $\theta$ has period $p>0$ and $y\leq0$, choose $k\in\nat$ such that $y+kp>0$. Then $\theta(y+\delta)=\theta(y+kp+\delta)\leq e^{\alpha\delta}\theta(y+kp)=e^{\alpha\delta}\theta(y)$.

{\rm ``(ii)$\Rightarrow$(iii)'':} Write $z=y-\delta$ then $\theta(y-\delta)=\theta(z)\geq e^{-\alpha\delta}\theta(z+\delta)=e^{-\alpha\delta}\theta(y)$.

{\rm ``(iii)$\Rightarrow$(iv)'':} For $0<s<t$ write $t=e^y$ and $s=e^{y-\delta}$ for some $y\in\rr$ and $\delta>0$. Then $s^{-\alpha}\theta(\log s)=e^{-\alpha(y-\delta)}\theta(y-\delta)\geq e^{-\alpha y}\theta(y)=t^{-\alpha}\theta(\log t)$.

{\rm ``(iv)$\Rightarrow$(i)'':} Write $y=\log t$ and $\delta=\log\gamma$ for some $t,\gamma>1$. Then $\theta(y+\delta)=\theta(\log(\gamma t))=(\gamma t)^{\alpha}(\gamma t)^{-\alpha}\theta(\log(\gamma t))\leq(\gamma t)^{\alpha}t^{-\alpha}\theta(\log t)=e^{\alpha\delta}\theta(y)$.

{\rm ``(ii)$\Rightarrow$(v)'':}  For fixed $y\in\rr$ we have $\frac{\theta(y+\delta)-\theta(y)}{\delta}\leq\frac{e^{\alpha\delta}-1}{\delta}\,\theta(y)$ for all $\delta>0$ and for differentiable $\theta$ as $\delta\downarrow0$ we get $\theta'(y)\leq\alpha\,\theta(y)$.

{\rm ``(v)$\Rightarrow$(iv)'':}  We have $\frac{d}{dt}(t^{-\alpha}\theta(\log t))=t^{-\alpha-1}(-\alpha\theta(\log t)+\theta'(\log t))\leq0$.
\end{proof}
Now we derive a formula for higher order derivatives of inverse functions used in Section 3 for which we couldn't find a suitable reference. 
\begin{lemma}
Let $A,B\subseteq\rr$ be open and let $f:A\to B$ be an invertible $C^n(A)$-function for some $n\in\nat$ such that $f'(x)\not=0$ for all $x\in A$. Then $f^{-1}:B\to A$ is of class $C^n(B)$ and for $n\geq2$ we have
\begin{align*}
(f^{-1})^{(n)}(x)=-\,\frac{1}{f'(f^{-1}(x))} & \sum_{k_1+2k_2+\cdots+(n-1)k_{n-1}=n}\frac{n!}{k_1!\cdots k_{n-1}!}\,f^{(k_1+\cdots+k_{n-1})}(f^{-1}(x))\\
& \qquad\qquad\cdot\prod_{j=1}^{n-1}\left(\frac{(f^{-1})^{(j)}(x)}{j!}\right)^{k_j}.
\end{align*}
\end{lemma}
\begin{remark}
For $n=1$ it is well known that $(f^{-1})'(x)=\frac{1}{f'(f^{-1}(x))}$. For $n=2$ we have $k_1=2$ in Lemma A.3 so that
$$(f^{-1})''(x)=-\,\frac{1}{f'(f^{-1}(x))}\,f''(f^{-1}(x))\cdot((f^{-1})'(x))^2=-\,\frac{f''(f^{-1}(x))}{(f'(f^{-1}(x)))^3}.$$
Iterating this procedure leads to a formula for $(f^{-1})^{(n)}$ given in \cite{ZL} not involving the derivatives of lower order $(f^{-1})',\ldots,(f^{-1})^{(n-1)}$. 
\end{remark}
\begin{proof}[Proof of Lemma A.2]
We use Fa\`a di Bruno's formula of higher order chain rule 
\begin{align*}
\frac{d^n}{dx^n}\,f(g(x))= & \sum_{k_1+2k_2+\cdots+nk_n=n}\frac{n!}{k_1!\cdots k_{n}!}\,f^{(k_1+\cdots+k_{n})}(g(x))\prod_{j=1}^{n}\left(\frac{(g^{(j)}(x)}{j!}\right)^{k_j}.
\end{align*}
For $j=n$ we must have $k_n=1$ and $k_1=\cdots=k_{n-1}=0$, otherwise $k_n=0$ if $k_j\geq1$ for some $j\in\{1,\ldots,n-1\}$. Hence we get
\begin{align*}
\frac{d^n}{dx^n}\,f(g(x))= & \sum_{k_1+2k_2+\cdots+(n-1)k_{n-1}=n}\frac{n!}{k_1!\cdots k_{n}!}\,f^{(k_1+\cdots+k_{n-1})}(g(x))\prod_{j=1}^{n-1}\left(\frac{(g^{(j)}(x)}{j!}\right)^{k_j}\\
& \qquad\qquad+f'(g(x))\,g^{(n)}(x).
\end{align*}
If $g=f^{-1}$, we know that $\frac{d^n}{dx^n}\,f(f^{-1}(x))=0$ for $n\geq2$ which directly leads to the stated formula for $(f^{-1})^{(n)}$.
\end{proof}

\bibliographystyle{plain}

\end{document}